\documentclass{amsart}
\usepackage{amssymb}
\usepackage{graphicx}
\newtheorem{comp prob}[thm]{Complementarity Problem}

\def\({\left(}
\def\){\right)}

\newtheorem{lema}{Lemma}[section]

\newtheorem*{teorema*}{Theorem}

\newtheorem{remark}[lema]{Remark}

\newtheorem{example}[lema]{Example}

\newtheorem{complementarity problem}[lema]{Complementarity Problem}

\newtheorem{corollary}[lema]{Corollary}
\newtheorem{theorem}[lema]{Theorem}

\newtheorem{complementarita prolema}[lema]{complementarity problem}

\newtheorem{definition}[lema]{Definition}

  {\par \hfill \fbox{}}
  {\par \hfill \fbox{}}
  {\par \hfill }
  {\par \hfill }

\def\beq{\begin{equation}}
\def\eeq{\end{equation}}

\def\epsilon{\varepsilon}

\begin{document}
\title{Li-Yorke and Expansive Composition operators on Lorentz spaces }
\author{Romesh kumar}
\address{Department of Mathematics, University of Jammu,
Jammu 180006, INDIA} 
\email{romeshmath@gmail.com}

\author{Rajat singh}
\address{Department of Mathematics, University of Jammu,
Jammu 180006, INDIA} \email{rajat.singh.rs634@gmail.com}

\thanks{The First author is supported by UGC Grant Under the Scheme of JRF}
\keywords{ Composition
operators, Li-Yorke chaotic, Lorentz spaces, multiplication operators, Expansive operators.\\ Mathematics Subject Classification: Primary 47A16, 47B33; Secondary 37D45, 37B05.}

\date{}

\begin{abstract}
In this paper, we investigate Li-Yorke composition operators and some of their variations on Lorentz spaces. Further, we also study expansive composition operators on these spaces. The work of the paper is essentially based on the work in \cite{BS}, \cite{NC15}, \cite{NC2018} and \cite{2022}.  
\end{abstract}

\maketitle

\section{Introduction and Preliminaries}
In the paper of Li and Yorke \cite{Li}, the concept of ``Chaos" was first introduced into mathematical literature in the context of interval map and became popular. Godefroy and Shapiro \cite{god} used Devaney's notion of chaos and were the first to introduce chaos into linear dynamics. Over the last two decades various authors have explored chaotic operators intensively. An operator on a Frechet space is hypercyclic and has a dense set of periodic points, then it is referred to be chaotic. Hypercyclic and chaotic operators are covered in depth in the books \cite{B1}, \cite{B2}, \cite{god}, \cite{KG1} and \cite{Shapiro}.\\
Some other essential concepts of chaos are Li-Yorke chaos, distributional chaos and specification property etc. see (\cite{NC13}, \cite{NC15} and \cite{Li}). Several authors have purposed various variations of these concepts. We will focus on Li-Yorke chaos and some of its variations. There are several intriguing Li-Yorke chaotic results for operators on Banach space in \cite{NC15}. N. C. Bernardes Jr et al. extended the major results of \cite{NC15} about Li-Yorke  chaos to the Frechet space setting and further for operators on $L^{p}$ space. The purpose of this note is to look into the concept of Li-Yorke chaos and some of its variations for Lorentz spaces framework. For more details on Lorentz spaces one can see (\cite{BS}, \cite{LOR}) and references therein. For more details on Li-Yorke we refer to \cite{NC13}, \cite{NC15} and \cite{Li} and reference therein.\\ 
The paper is structured as follows: Section 1 is introductory and we cite certain definitions and results which will be used throughout this paper. In Section 2 and Section 3, we explore the Li-Yorke composition operators  and discuss expansive composition operators on Lorentz spaces respectively. 

We assume that $X=\( X,\mathbb{A},\mu\)$ be a measure space with $\mu(X)\neq 0$. Let $\tau:X \to X$ be a measurable non-singular transformation $(i.e., \mu(\tau^{-1}(A))=0~\mbox{for each}~A\in \mathbb{A}~\mbox{whenever}~\mu(A)=0)$.\\
We define the distribution function $\mu_{g}$ of $g$, for $\lambda\geq 0$ as
$$\mu_{g} = \mu \(\{ x \in X:|g(x)|> \lambda \}\).$$
The non-increasing rearrangement of $g$ is
$$ g^{*}(t)= \inf \{ \lambda >0:\mu_{g}(\lambda)\leq t\}=\sup\{\lambda>0:\mu_{g}(\lambda)>t\}.$$
 The norm of the measurable function $g$ is defined as
$$||g||_{pq}= \begin{cases}
\left \{ \frac{q}{p}\int_{0}^{\infty}(t^{\frac{1}{p}}g^{**}(t))^{q}\frac{dt}{t}\right\}^{\frac{1}{q}}, &\text{if $1<p<\infty,~1\leq q<\infty$}\\
\displaystyle\sup_{t>0} t^{\frac{1}{p}}g^{**}(t), &\text{if $1<p\leq\infty,~q=\infty$;}\\
\end{cases}
$$ where $1<p\leq \infty,~ 1\leq q \leq \infty$.\\
The Lorentz space $L^{pq}(X)$, $1<p\leq \infty,~ 1\leq q \leq \infty$ are defined as
$$L^{pq}(X) = \{ g\in L(\mu): ||g||_{pq}<\infty \}.$$
Note that the Lorentz spaces are the Banach spaces for $1\leq q \leq p<\infty,~or~p=q=\infty$ and by using \cite[Page 251]{BS} 
\begin{eqnarray*}
||\chi_{A}||_{pq}^{q} &=& \frac{q}{p}\int_{0}^{\infty}(t^{\frac{1}{p}}\chi_{A}^{**}(t))^{q}\frac{dt}{t}\\
                      &=& p^{'}(\mu(A))^{\frac{q}{p}}
\end{eqnarray*}
where $\frac{1}{p}+\frac{1}{p^{'}}=1$.\\
We now define $C_{\tau}$ as the the linear transformation on $ L^{pq}(X),~1<p\leq \infty,~1\leq q \leq \infty$ into the linear space of all complex valued measurable function on measure space $\( X,\mathbb{A},\mu\)$ by $C_{\tau}g=g \circ \tau,~\forall g\in L^{pq}(X).$ Here the non-singularity of $\tau$ ensures that the operator is well defined in this case. If $C_{\tau}$ maps the $L^{pq}(X)$ into itself, then we call it composition operator on Lorentz space induced by $\tau$.
Let $\theta$ be a complex valued measurable function defined on $X$. We define the mapping $M_{\theta}:g \to \theta. g$, a multiplication operator induced by $\theta$. For composition operator on different function spaces see \cite{RK} and \cite{Shapiro} and reference therein. \\

First of all we recall the basic definitions which will be used for further research.
\begin{definition}\cite[Page 1]{Berm}
A continuous map $g:(M,d) \to (M,d)$  is said to be Li-Yorke chaotic if there exists an uncountable scrambled set $S\subset M$  such that each pair of distinct points $p,q \in S$ is a Li-Yorke pair for $g$ i.e.,
$$\displaystyle\lim_{n\to \infty}\inf d(g^{n}(p), g^{n}(q)) = 0 ~\mbox{and}~ \displaystyle\lim_{n\to \infty}\sup d(g^{n}(p), g^{n}(q)) > 0 .$$
where $(M,d)$ is a metric space.
\end{definition}
We say that $g$ is densely (generically) Li-Yorke chaotic whenever $S$ can be chosen to be dense (residual) in $M$. 
\begin{definition}\cite[Page 47]{B2}
\begin{itemize}
\item[(a)] If  $T$ is a linear operator and a vector $z \in X$, then we say that $z$ is an irregular vector for $T$ if 
$$\displaystyle\lim_{n \to \infty} \inf||T^{n}z||=0~\mbox{and}~\displaystyle\lim_{n \to \infty} \sup||T^{n}z||=\infty.$$
\item[(b)] If $T$ is a linear operator and a vector $z \in X$, then we say that $z$ is semi-irregular vector for $T$ if 
 $$\displaystyle\lim_{n \to \infty} \inf||T^{n}z||=0~\mbox{and}~\displaystyle\lim_{n \to \infty} \sup||T^{n}z||>0.$$ 
\end{itemize}
\end{definition}

Following result gives the equivalent conditions for any continuous linear operator $T$ on any Banach space to be Li-Yorke. 
\begin{theorem}\cite[Theorem 9]{NC15}
If $T \in L(X)$, then the following are equivalent
\begin{itemize}
\item[(i)] $T$ is Li-Yorke chaotic.
\item[(ii)] $T$ admits a semi-irregular vector. 
\item[(iii)] $T$ admits irregular vector. 
\end{itemize} 
\end{theorem}

\begin{definition}\cite{2022}
Let $T\in L(X)$ be linear operator. Then
\begin{itemize}
\item[(a)] $T$ is said to be (positively) expansive if for all $x\in S_{X}$ there exists $n \in \mathbb{Z} (n \in \mathbb{N})$ such that $||T^{n}x|| \geq 2$, where $S_{X}=\{x\in X:||x||=1 \}$.
\item[(b)] $T$ is (positively) uniformly expansive if there exists $n \in \mathbb{N}$ such that for all $x \in S_{X}$, $||T^{n}x|| \geq 2 ~\mbox{or}~||T^{-n}x|| \geq 2$ (for all $x \in S_{X},~||T^{n}x|| \geq 2$).
\end{itemize} 
\end{definition}

\begin{theorem}\cite[Proposition 19]{NC2018}\label{Thrm1.5}
Let $X$ be a Banach space and $T$ be operator on $X$. Then\\
(a)  $\displaystyle\sup_{n\in \mathbb{N}} ||T^{n}x||=\infty$ if and only if $T$ is positively expansive,  for each $0 \neq x\in  X$.\\
(b) $\displaystyle\lim_{n\in \infty} ||T^{n}x||=\infty$ uniformly on $S_{X}$ if and only if $T$ is uniformly positively expansive.\\ 
If $T$ is invertible, then\\
(c) $\displaystyle\sup_{n\in \mathbb{Z}} ||T^{n}x||=\infty$ if and only if $T$ is expansive, for each $ 0 \neq x\in X$.\\
(d) $S_{X} = A \cup B$ where $\displaystyle\lim_{n\in \infty} ||T^{n}x||=\infty$ uniformly on A and $\displaystyle\lim_{n\in \infty} ||T^{-n}x||=\infty$ uniformly on B if and only if $T$ is uniformly expansive.
\end{theorem}

\section{Li-Yorke composition operator on Lorentz space}
In this section, we have proved a necessary and sufficient condition for composition operator to be Li-Yorke.
\begin{theorem}\label{thrm1}
Let $\( X,\mathbb{A},\mu\)$ be a measure space and $\tau:X \to X$ be a non-singular measurable transformation. Then composition operator $C_{\tau}$ on $L^{pq}(X)~1\leq p<\infty,~1\leq q\leq \infty$ is Li-Yorke chaotic iff there is an increasing sequence of positive integers and non-empty family of measurable sets $A_{i}$ of finite positive measure $\mu$ such that 
\begin{itemize}
\item[(i)] $\displaystyle\lim_{j\to \infty}\mu(\tau^{-\alpha_{j}}(A_{i}))=0,~\forall~i \in \mathbb{N}.$
\item[(ii)] $\sup\left\{\frac{\mu \circ \tau^{-n}(A_{i})}{\mu(A_{i})}:i \in I,n\in\mathbb{N} \right\}=\infty.$
\end{itemize}  
\end{theorem}

\begin{proof}
Suppose $C_{\tau}$ is Li-Yorke chaotic and $g \in L^{pq}(X)$ be an irregular vector for $C_{\tau}$. Now, let the measurable set $A_{i}=\{x\in X:2^{i-1}<|g(x)|<2^{i} \}$ and $I=\{i \in \mathbb{Z}:\mu(A_{i})>0\}.$ Then, $0<\mu(A_{i})<\infty.$ As $g$ be an irregular vector for $C_{\tau}$, so there is an increasing sequence of positive number $\{\alpha_{j}\}_{j\in\mathbb{N}}$ such that $\displaystyle\lim_{j\to \infty}||C_{\tau^{\alpha_{j}}}g||_{pq}=0$. This implies that (i) holds.

 Now, suppose that the condition (ii) does not holds. Then there is a positive constant $M<\infty$ such that $$\mu \circ \tau^{-n}(A_{i}) \leq M\mu(A_{i}),~\mbox{whenever}~i\in \mathbb{Z},n\in \mathbb{N}.$$
Thus, for each $n \in \mathbb{N},~t \geq 0$
\begin{eqnarray*}
(g \circ \tau^{n})(t) &=& \displaystyle\sum_{n \in\mathbb{N}} \inf\{s>0:\mu\{x\in X:|g(\tau^{n}(x))|>s\}\leq t\} \\
                      &=& \displaystyle\sum_{n \in\mathbb{N}} \inf\{s>0:\mu \tau^{-n}\{x\in X:|g(x)|>s\}\leq t\}\\
                      &\leq& \displaystyle\sum_{n \in\mathbb{N}} \inf\{s>0:M \mu \{x\in X:|g(x)|>s\}\leq t\} \\
											&\leq& \displaystyle\sum_{n \in\mathbb{N}} \inf\left\{2^{i-1}>0:\mu\{x\in X:|g(x)|>2^{i-1}\}\leq \frac{t}{M} \right\}\\
                      &\leq& \displaystyle\sum_{n \in\mathbb{N}} g^{*}\(\frac{t}{M}\).
\end{eqnarray*}
 Consequently we get
\begin{eqnarray*}
(g \circ \tau^{n})^{**}(t) &\leq& \displaystyle\sum_{n \in\mathbb{N}} g^{**}\(\frac{t}{M}\).
\end{eqnarray*}
Thus for $q\neq \infty$, we have
\begin{eqnarray*}
||C_{\tau^{n}}g||_{pq} &=& [ \frac{q}{p}\int_{0}^{\infty}(t^{\frac{1}{p}}(C_{\tau^{n}}g)^{**}(t))^{q}\frac{dt}{t}]^{\frac{1}{q}}\\
                       &\leq& \displaystyle\sum_{n \in\mathbb{N}} [\frac{q}{p} \int_{0}^{\infty}(t^{\frac{1}{p}}g^{**}\(\frac{t}{M}\))^{q}\frac{dt}{t}]^{\frac{1}{q}}\\
											&=& \displaystyle\sum_{n \in\mathbb{N}} [ \frac{q}{p}\int_{0}^{\infty}(M^{\frac{1}{p}}t^{\frac{1}{p}}g^{**}(t))^{q}\frac{dt}{t}]^{\frac{1}{q}}\\
											&=& M^{\frac{1}{p}} ||g||_{pq}.
\end{eqnarray*}
Also for $q=\infty$.
\begin{eqnarray*}
||C_{\tau^{n}}g||_{p\infty}	&=& \sup_{0<t<\infty} t^{\frac{1}{p}}(C_{\tau^{n}}g)^{**}(t)\\
                            &\leq& \sup_{0<t<\infty} t^{\frac{1}{p}}(g)^{**}\(\frac{t}{M}\)\\ 										
                            &\leq& M^{\frac{1}{p}} ||g||_{p\infty}.
														\end{eqnarray*}
i.e., $C_{\tau}$-orbit of $g$ is bounded, which is contradiction to the fact that $C_{\tau}$-orbit of $g$ is unbounded.\\
Conversely, Suppose condition (i) and (ii) holds and let $Y =\{ \chi_{A_{i}}:i\in I\}$ be a closed linear span in $L^{pq}(X)$. Then set $R_{1}$ of all vectors $g$ in $Y$, where $C_{\tau}$-orbit has sub-sequence converging to zero is residual in $Y$ because of condition (i).\\
Now, for $i\in I$, let $g_{i}=\frac{1}{(p')^{\frac{1}{p}}(\mu(A_{i}))^{\frac{1}{p}}}.\chi_{A_{i}} \in Y.$ Then, 
 $$||g_{i}||=1~ \mbox{and}~||C_{\tau^{n}}g_{i}||_{pq}= \frac{\mu(\tau^{-n}(A_{i}))}{\mu(A_{i})}.$$
Thus, by conditions (ii), $\displaystyle\sup_{n \in \mathbb{N}} ||C_{\tau^{n}}|_{Y}||=\infty$ and so by using the Banach Steinhaus theorem, the set $R_{2}$ of all vectors $g$ in $Y$ whose $C_{\tau}$-orbit is unbounded is residual in $Y$. Also, as $g \in R_{1}\cap R_{2}$ is an irregular vector for $C_{\tau}$, we conclude that $C_{\tau}$ is Li-Yorke chaotic.														
\end{proof}

\begin{corollary}\label{cor1}
 If $\tau$ is injective, then composition operator $C_{\tau}$ is Li-Yorke chaotic if there exists a measurable set $A$ of finite positive $\mu$-measure such that 
\begin{itemize}
\item[(a)] $\displaystyle\lim_{n \to \infty}\inf \mu(\tau^{-n}(A))=0$,\\
\item[(b)] $\sup\{\frac{\mu(\tau^{n}(A))}{\mu(\tau^{m}(A))}:n\in\mathbb{Z},m\in I,n<m\}=\infty.$
\end{itemize}
\end{corollary}
\begin{remark}
If $\tau$ is not injective in above Corollary \ref{cor1}, then $C_{\tau}$ need not Li-Yorke chaotic. Here is an example:
\end{remark}
\begin{example}
Let us consider $\mathbb{A}=P(X)$ and $X=(\mathbb{Z} \times \{0\})\cup (\mathbb{N}\times \mathbb{N})$. The bimeasurable map $\tau:X \to X$ be
$$\tau(i,0)=(i+2,0)~\mbox{and}~\tau(n,j)=(n,j-1)~i\in \mathbb{Z}~\mbox{and}~n,j\in \mathbb{N}.$$
The measure $\mu:\mathbb{A} \to [0,\infty)$ be defined by 
$$\mu(\{(i,0)\})=\frac{1}{3^{|i|}}~\mbox{and}~ \mu(\{(n,j)\})=\left\{\begin{array}{cc}
                \frac{1}{3^{n-j}},  ~~1\leq j<n  \\
								1,	 ~~j\geq n
									\end{array}\right.$$ 
If $A=\{(0,0)\}$, then clearly conditions of Corollary \ref{cor1} are satisfies. But however, if $A\in \mathbb{A}$ is non-empty and satisfies condition (i) of Theorem \ref{thrm1} then $A_{i}\subset \{(k,0):k\leq 0\}$ and so
		$$\displaystyle\sup_{n \in \mathbb{N}}\frac{\mu(\tau^{-n}(A_{i}))}{\mu(A_{i})}=\frac{1}{9}.$$
Thus, by Theorem \ref{thrm1}, $C_{\tau}$ is not Li-Yorke Chaotic.
\end{example}

\begin{theorem}
If $\mu$ is finite and $\tau$ is injective, then the following are equivalent:
\begin{itemize}
\item[(i)] $C_{\tau}$ is Li-Yorke chaotic.
\item[(ii)] there exists $g \in L^{pq}(X)$ such that $g \neq 0$ and $\displaystyle\lim_{n\to \infty}\inf||C_{\tau^{n}}g||_{pq}=0$.
\item[(iii)] there exist $A \in \mathbb{A}$ such that $\mu(A)>0$ and $\displaystyle\lim_{n\to \infty} \mu(\tau^{-n}(A))=0$.
\item[(iv)] there exist $A \in \mathbb{A}$ such that $\mu(A)>0$ and $\displaystyle\lim_{n\to \infty} \mu(\tau^{n}(A))=0$.
\item[(v)] there exist $A \in \mathbb{A}$ such that $\mu(A)>0$, $\displaystyle\lim_{n\to \infty} \inf \mu(\tau^{-n}(A))=0$ and $\displaystyle\lim_{n\to \infty} \inf \mu(\tau^{n}(A))=0$.
\item[(vi)] there exist $A \in \mathbb{A}$ such that $\mu(A)>0$, $\displaystyle\lim_{n\to \infty}\inf \mu(\tau^{-n}(A))=0$ and $\displaystyle\lim_{n\to \infty} \sup \mu(\tau^{-n}(A))>0$.
\item[(vii)] $C_{\tau}$ admits a characteristic function as a semi-irregular vector.
\end{itemize}
\end{theorem}

\begin{proof}
$(i)\implies (ii)$
Since $C_{\tau}$ is Li-Yorke chaotic. Then it admits a semi-irregular vector $g \in L^{pq}(X)$. Thus, by definition of semi-irregularity, $g \neq 0$ and $\displaystyle\lim_{n\to \infty}\inf||C_{\tau^{n}}g||=0$.\\
$(ii) \implies (iii)$ Suppose $g$ satisfies the condition (ii). Then there exists $c>0$ such that $A=\{x\in X:|g(x)|>c \}$. Clearly, $A$ is measurable and $\mu(A)>0$. 
Hence, 
\begin{eqnarray*}
||C_{\tau^{k}}g||_{pq}^{p} &=& \frac{q}{p}\int_{0}^{\infty}(t^{\frac{1}{p}}(C_{\tau^{n}}f)^{*}(t))^{q}\frac{dt}{t} \\
                           &\geq& \frac{q}{p}\int_{0}^{\mu(\tau^{-n}(A))}(t^{\frac{1}{p}} c)^{q} \frac{dt}{t} \\
													 &\geq& c^{p} \frac{q}{p}\int_{0}^{\mu(\tau^{-n}(A))}(t^{\frac{1}{p}})^{q} \frac{dt}{t} \\
                           &\geq& c^{p}. \mu(\tau^{-n}(A)).
\end{eqnarray*}
By using (ii), we see that $\displaystyle\lim_{k\to \infty}\inf \mu(\tau^{-k}(A))=0.$\\
The implication $(iii) \implies (iv)$, $(iv) \implies (v)$ and $(v) \implies (vi)$ will follows as in \cite{NC19}.\\
$(vi)\implies (vii)$ By taking $g=\chi_{A}$ for some $A \in \mathbb{A}$, we have
$$||C_{\tau^{k}}g||_{pq}^{q}=||C_{\tau^{k}}\chi_{A}||_{pq}^{q}=||\chi_{\tau^{-k}(A)}||_{pq}^{q}=p^{'}(\mu(\tau^{-k}(A)))^{\frac{q}{p}}.$$
so, (vi) and (vii) are equivalent properties.\\
$(vii) \implies (i)$ is obvious, because the existence of semi-irregular vector itself implies that $C_{\tau}$ is Li-Yorke chaotic. 

\end{proof}

\begin{theorem}
Let $\( X,\mathbb{A},\mu\)$ be a $\sigma$-finite measure space and $\tau:X \to X$ be a non-singular measurable transformation. The $C_{\tau}$ is then topological transitive if and only if $C_{\tau}$ is densely Li-Yorke chaotic. 
\end{theorem}

\begin{proof}
Since in \cite{NC15}, it has been established that the continuous linear operator admits a dense set of irregular vectors for separable Banach space if and only if it admits a dense set of irregular vectors. By (\cite[ Remark 22]{NC15}), if an operator is topologically transitive, it is densely Li-Yorke chaotic. From this direct part follows because Lorentz space are separable.\\
Conversely, let us supposed that composition operator $C_{\tau}$ is densely Li-Yorke chaotic and let $\epsilon \in (0,min\{1,\mu(X)\}).$ Then there is an irregular vector $g$ for $C_{\tau}$ such that $$||g-\chi_{X}||_{pq}^{p}<\epsilon.$$
Taking $A=\{x \in X:|g(x)-1 |<\epsilon\}$. Then $\mu(X\setminus A)<\epsilon$.\\
Note $g$ and $\sum g(A)\chi_{A}$ are $\mu$-a.e. For each $k \in \mathbb{N}$, define $C_{\tau^{k}}$ by
$$C_{\tau^{k}}g = \sum_{k\in \mathbb{N}} g(A)\chi_{\tau^{-k}(A)}.$$
Then for each $\lambda=1-\epsilon$, we have 
\begin{eqnarray*}
\mu_{C_{\tau^{k}}g}(\lambda) &\leq& \sum_{k \in \mathbb{N},~|g(A)|>\lambda} \mu(\tau^{-k}(A))\\
                             &\leq& \displaystyle \sup_{k \in\mathbb{N}}\frac{\mu(\tau^{-k}(A))}{\mu(A)}\sum_{|g(A)|>\lambda} \mu(A)\\
														 &\leq& \displaystyle \sup_{k \in\mathbb{N}}\frac{\mu(\tau^{-k}(A))}{\mu(A)} \mu_{g}(\lambda).
\end{eqnarray*}
and so we obtain
$$||C_{\tau^{k}}g||^{p}_{pq} \leq \displaystyle \sup_{k \in\mathbb{N}}\frac{\mu(\tau^{-k}(A))}{\mu(A)}||g||_{pq} \leq \epsilon \displaystyle \sup_{k \in\mathbb{N}}\frac{\mu(\tau^{-k}(A))}{\mu(A)} .$$
Thus $\displaystyle\lim_{n \to \infty}\inf \mu(\tau^{-k}(A))=0,$ because $g$ is an irregular vector for $C_{\tau}$. By using \cite[Lemma 2.1]{NC15}, there exist a measurable set $W\subset A$ such that
$$\mu(X\setminus W)<\epsilon~\mbox{and}~\displaystyle\lim_{n \to \infty}\inf \mu(\tau^{k}(W))=0$$ 
So, $C_{\tau^{k}}$ is topologically transitive.
\end{proof}

In the next theorem, we discuss the Li-Yorke multiplication operators on Lorentz space.
\begin{theorem}
Multiplication operator $M_{\theta}$ is not Li-Yorke chaotic on $L^{pq}(X)$.
\end{theorem}

\begin{proof}
Suppose on the contrary that $M_{\theta}$ is Li-Yorke chaotic. Then it admits a irregular vector $g \in L^{pq}(X)$. Let $\(n_{k}\)$ be increasing sequence of positive integers such that $\mu((M_{\theta})^{n_{k}}g) \to 0$ in $L^{pq}(X)$. Then $\mu((\theta(x))^{n_{k}}g(x)) \to 0,~\forall ~x \in X$. \\
Let $E=\{x \in X:|\theta(x)|<1 \}$. Then, clearly $E$ is measurable set with positive measure. The distribution function for $M_{\theta}$ is:
\begin{eqnarray*}
\mu_{M_{\theta}g}(s) &=& \mu\{x \in X:|M_{\theta}g(x)|>s\} \\
                   &=& \mu\{x \in X:|\theta(x)g(x)|>s\} \\
									&\leq& \mu\{ x\in X:|g(x)|>s\}. 
\end{eqnarray*}
Then for $t\geq 0$,
\begin{eqnarray*}
(M_{\theta}g)^{*}(t) &=& \inf\{x \in X: \mu_{M_{\theta}g}(s)\leq t \} \\
                   &\leq& \inf\{x \in X: \mu\{x \in X:|g(x)|>s \}\leq t\} \\
									 &\leq& g^{*}(t).
\end{eqnarray*}									
Thus,
$(M_{\theta}g)^{**}(t)\leq g^{**}(t).$
Thus, for $1<p\leq\infty,~1\leq q<\infty$
\begin{eqnarray*}
||M_{\theta}g||_{pq}^{q} &=& \frac{q}{p}\int_{0}^{\infty} (t^{\frac{1}{p}}(M_{\theta}g)^{**}(t))^{q}\frac{dt}{t} \\
                       &\leq& \frac{q}{p}\int_{0}^{\infty} (t^{\frac{1}{p}}g^{**}(t))^{q}\frac{dt}{t} \\
											&\leq& ||g||_{pq}^{q}.
											\end{eqnarray*}
Hence, for $q=\infty$
\begin{eqnarray*}
||M_{\theta}g||_{p\infty}^{q} &\leq& \sup_{0<t<\infty} t^{\frac{1}{p}}(M_{\theta}g)^{**}(t) \\
                        &=& ||g||_{p\infty}^{q},
\end{eqnarray*}											
which contradicts our assumption that $g$ is irregular vector, which completes the proof.
									
\end{proof}

\section{Expansive composition operators on Lorentz space}
In this section, we give a necessary and sufficient condition for composition operators to be expansive and uniformly expansive on $L^{pq}(X,\mathbb{A},\mu)$.

\begin{theorem}
Let $(X,\mathbb{A},\mu)$ be a $\sigma$-finite measure space and $\tau$ be a non-singular measurable transformation. Then $C_{\tau}$ is positively expansive iff for each $A \in \mathbb{A}$ with positive measure, $\displaystyle\sup_{n \in \mathbb{Z}}\mu(\tau^{-n}(A)) =\infty.$
\end{theorem}
\begin{proof}
First of all suppose that $C_{\tau}$ is expansive. Then by \cite[Proposition 19]{NC2018},
$$\displaystyle\sup_{n \in \mathbb{Z}} ||C_{\tau}^{n}g||_{pq}=\infty,~\mbox{for each}~g \in L^{pq}(X)\setminus \{0\}.$$
Let $A\in \mathbb{A}$ with $\mu(A)>0$ and taking $g=\chi_{A}$. Then for each $n\in \mathbb{Z}$, the non-increasing re-arrangement of $\chi_{A}$ is
$$\chi_{A}^{*}(t)=\chi_{[0,\mu(A))}(t).$$
Therefore,
\begin{eqnarray*}
||\chi_{A}||_{pq}^{q} &=& \frac{q}{p}\int_{0}^{\infty}(t^{\frac{1}{p}}\chi_{A}^{**}(t))^{q}\frac{dt}{t}\\
                      &=& p^{'}(\mu(A))^{\frac{q}{p}}.
\end{eqnarray*}

So, for each $n\in \mathbb{Z} ~\mbox{and for}~ 1\leq q<\infty ,$
\begin{eqnarray*}
	||C_{\tau}^{n}\chi_{A}||_{pq}^{p} &=& ||\chi_{\tau^{-n}(A)}||_{pq}^{p}\\
							                      &=& \sum_{n \in \mathbb{Z}}\mu(\tau^{-n}(A)).
\end{eqnarray*}
For $q=\infty, 1<p\leq \infty$, we have
\begin{eqnarray*}
||C_{\tau}^{n}\chi_{A}||_{pq}^{p} &=& \displaystyle\sup_{t\geq \mu(A)} t^{\frac{1}{p}}\chi_{A}^{**}(t) \\
                                  &=& \displaystyle\sup_{t\geq \mu(A)} \mu(\tau^{-n}(A))
																	\end{eqnarray*}
and so we get, $\displaystyle\sup_{n\in \mathbb{Z}}\mu(\tau^{-n}(A))=\infty$. This proves the direct part.\\
For the converse part, suppose $\displaystyle\sup_{n \in \mathbb{Z}}\mu(\tau^{-n}(A))=\infty$ for each $A\in \mathbb{A}$ with $\mu(A)>0$. Let $g\in L^{pq}(X)\setminus\{0\}$. Then there exist an $h>0$ such that the set $A^{'}=\{x\in X:|g(x)|>h \}$ has positive measure.\\
Now, for each $n \in \mathbb{Z}$
\begin{eqnarray*}
||C_{\tau}^{n}g||_{pq}^{p} &=& \frac{q}{p}\int_{0}^{\infty}(t^{\frac{1}{p}}(C_{\tau^{n}}g)^{*}(t))^{q}\frac{dt}{t} \\
                           &\geq& \frac{q}{p}\int_{0}^{\mu(\tau^{-n}(A^{'}))}(t^{\frac{1}{p}} h)^{q} \frac{dt}{t} \\
													 &\geq& h^{p} \frac{q}{p}\int_{0}^{\mu(\tau^{-n}(A^{'}))}(t^{\frac{1}{p}})^{q} \frac{dt}{t} \\
                           &\geq& h^{p}. \mu(\tau^{-n}(A^{'})).
\end{eqnarray*}
This implies that $	\displaystyle\sup_{n \in \mathbb{Z}}||C_{\tau}^{n}g||_{pq}=\infty.$ Thus, it follows that $C_{\tau}$ is expansive.														\end{proof}

\begin{corollary}
Let $(X,\mathbb{A},\mu)$ be a $\sigma$-finite measure space and $\tau$ be a non-singular measurable transformation. Then $C_{\tau}$ is positively expansive iff for each $A \in \mathbb{A}$ with positive measure, $\displaystyle\sup_{n \in \mathbb{N}}\mu(\tau^{-n}(A)) =\infty.$
\end{corollary}
The proof will directly follows from the above theorem by replacing $\mathbb{Z}$ by $\mathbb{N}$.

\begin{theorem}\label{Thrm expansive}
 Let $(X,\mathbb{A},\mu)$ be $\sigma$-finite measure space and $\tau$ be the non-singular measurable transformation. Let $A_{n}$ be all the atoms of $X$ and
assume that $\mu(A_{n}) = a_{n} > 0, $ for each n. Then $C_{\tau}$ is uniformly positively expansive iff $$\displaystyle \lim_{n \to \infty}\frac{\mu(\tau^{-n}(A))}{\mu(A)}=\infty$$ uniformly with respect to $A\in \mathbb{A^{+}}$, where $\mathbb{A^{+}}=\{A \in \mathbb{A}: 0< \mu(A)<\infty \}.$ 
\end{theorem}

\begin{proof}
Suppose $C_{\tau}$ is uniformly positively expansive. Then by Theorem \ref{Thrm1.5} $$\displaystyle \lim_{n \to \infty}||C_{\tau}^{n}g||_{pq}=\infty, ~\mbox{uniformly on}~S_{L^{pq}(X)}.$$\\
Let $g=\frac{\chi_{A}}{\mu(A)^{\frac{1}{p}}}$, for all $A\in \mathbb{A}^{+}$. Then
\begin{eqnarray*}
||C_{\tau}^{n}g||_{pq}^{p} &=& ||C_{\tau}^{n} \frac{\chi_{A}}{\mu(A)^{\frac{1}{p}}}||_{pq}^{p}\\
                           &=& \frac{|| \chi_{\tau^{-n}(A)}||_{pq}^{p}}{\mu(A)} \\
													&=& (p^{'})^{\frac{p}{q}} \frac{\mu(\tau^{-n}(A))}{\mu(A)}
\end{eqnarray*} 
and so, $$\infty=\displaystyle\lim_{n\to \infty}||C_{\tau}^{n}g||_{pq}^{p}=(p^{'})^{\frac{p}{q}} \displaystyle\lim_{n\to \infty}\frac{\mu(\tau^{-n}(A))}{\mu(A)}.$$
This implies that, $\displaystyle\lim_{n\to \infty}\frac{\mu(\tau^{-n}(A))}{\mu(A)}= \infty$, uniformly on $A\in \mathbb{A}^{+}.$\\
For the converse part, according to Theorem\ref{Thrm1.5}, it will be enough to show that $\displaystyle\lim_{n\to \infty} ||C_{\tau}^{n}g||_{pq}=\infty$, uniformly on $S_{L^{pq}(X)}$ for simple functions.\\
By the given conditions, let $M>0$, there exists $m\in \mathbb{N}$ such that for each atoms $A_{n},~\frac{\mu(\tau^{-n}(A_{n}))}{\mu(A_{n})}>M,~~\forall~n\geq m.$\\
Let $g \in S_{L^{pq}(X)}$ be simple functions i.e., $g =\sum g(A_{n})\chi_{A_{n}},$ where $(X,\mathbb{A},\mu)$ be atomic with atoms $A_{n}$. Note that $g$ and $\sum g(A_{n})\chi_{A_{n}}$ are equal $\mu$-a.e. Then for $n \geq m$ and $\lambda>0$, we have
\begin{eqnarray*}
\mu_{C_{\tau}^{n}g}(\lambda) &=& \displaystyle\sum_{n\geq m,|g(A_{n})|>\lambda} \mu(\tau^{-n}(A_{n}))\\
                             &\geq& \displaystyle\sum_{n\geq m,|g(A_{n})|>\lambda} M \mu(A_{n}) \\
														&=& M \displaystyle\sum_{n\geq m,|g(A_{n})|>\lambda}  \mu(A_{n}).
\end{eqnarray*}
Therefore,	$||	C_{\tau}^{n}g||_{pq}^{p} \geq M||g||_{pq}^{p} \geq M.$\\
Thus, for each $M>0$, there exist $n \geq m$ such that for each simple function $g \in S_{L^{pq}(X)}$,
$$||	C_{\tau}^{n}g||_{pq}^{p} \geq M,~\forall n \geq m$$
i.e., $\displaystyle\lim_{n\to \infty}||C_{\tau}^{n}g||_{pq}^{p}=\infty.$

\end{proof}

\begin{theorem}
Let $(X,\mathbb{A},\mu)$ be a $\sigma$-finite measure space and $\tau$ be a non-singular measurable transformation. Let $A_{n}$ be all the atoms of $X$ and
assume that $\mu(A_{n}) = a_{n} > 0, $ for each n. Then $C_{\tau}$ is uniformly expansive iff $\mathbb{A}^{+}$ can be splitted as $\mathbb{A}^{+}=\mathbb{A}_{B}^{+}\cup \mathbb{A}_{C}^{+}$ where 
$$\displaystyle \lim_{n \to \infty}\frac{\mu(\tau^{n}(A))}{\mu(A)}=\infty,~\mbox{uniformly on $\mathbb{A}_{B}^{+}$},$$
$$\displaystyle \lim_{n \to \infty}\frac{\mu(\tau^{-n}(A))}{\mu(A)}=\infty,~\mbox{uniformly on $\mathbb{A}_{C}^{+}$}.$$
\end{theorem}

\begin{proof}
Suppose $C_{\tau}$ is uniformly expansive. Then by part(d) of Theorem\ref{Thrm1.5}, $S_{L^{pq}(X)} =B \cup C,$ where 
$$\displaystyle\lim_{n\to \infty}||C_{\tau}^{n}g||_{pq}=\infty,~\mbox{uniformly on B}~\mbox{and}~\displaystyle\lim_{n\to \infty}||C_{\tau}^{-n}(g)||_{pq}=\infty,~\mbox{uniformly on C}.$$
This implies $\mathbb{A}^{+}=\mathbb{A}_{B}^{+}\cup \mathbb{A}_{C}^{+}$, where $\mathbb{A}_{B}^{+}=\{A \in \mathbb{A}^{+}:\frac{\chi_{A}}{(\mu(A))^{\frac{1}{p}}} \in B\}$ and $  \mathbb{A}_{C}^{+}=\{A \in \mathbb{A}^{+}:\frac{\chi_{A}}{(\mu(A))^{\frac{1}{p}}} \in C\}.$ So, by Theorem\ref{Thrm expansive}, we see that
$$\displaystyle\lim_{n \to \infty} \frac{\mu(\tau^{n}(A))}{\mu(A)}=\infty,~\mbox{uniformly on $\mathbb{A}_{B}^{+}$}~\mbox{and}~\displaystyle\lim_{n \to \infty} \frac{\mu(\tau^{-n}(A))}{\mu(A)}=\infty,~\mbox{uniformly on $\mathbb{A}_{C}^{+}$},$$ which proves the direct part.\\
In order to prove the converse part, it is sufficient to prove using again part(d) of Theorem\ref{Thrm1.5}, the existence of $B$ and $C$ such that $S_{L^{pq}(X)}=B \cup C,$ with $$\displaystyle\lim_{n \to \infty}||C_{\tau}^{n}g||_{pq}=\infty,~\mbox{uniformly on B}~\mbox{and}~\displaystyle\lim_{n \to \infty}||C_{\tau}^{-n}g||_{pq}=\infty,~\mbox{uniformly on C}.$$
By given hypothesis, for $M>0$, there exist $m \in \mathbb{N},$ such that for all functions of type $g=\frac{\chi_{A}}{(\mu(A))^{\frac{1}{p}}}$ with $A \in \mathbb{A}^{+}$, $||C_{\tau}^{n}g||_{pq}^{p} \geq M,~\mbox{or}~||C_{\tau}^{-n}g||_{pq}^{p}\geq M,~\forall n \geq m.$
We have proved it for simple function $g \in S_{L^{pq}(X)}$. Let $\hat{S}_{L^{pq}(X)}$ be the collection of all simple functions in $S_{L^{pq}(X)}$. First we find the two sets of simple function in $\hat{S}_{L^{pq}(X)}$, denoted by $\hat{B}$ and $\hat{C}$ such that one has $\hat{S}_{L^{pq}(X)}=\hat{B}\cup \hat{C}$, with 
$$\displaystyle\lim_{n \to \infty}||C_{\tau}^{n}g||_{pq}=\infty,~\mbox{uniformly on $\hat{B}$}~\mbox{and}~\displaystyle\lim_{n \to \infty}||C_{\tau}^{-n}g||_{pq}=\infty,~\mbox{uniformly on $\hat{C}$}.$$
By hypothesis, for $M>0$, there exist $\bar{n}\in \mathbb{N}$ such that for each $n \geq \bar{n},$
$$\frac{\mu(\tau^{n}(A))}{\mu(A)} >M,~\mbox{for each}~A \in \mathbb{A}_{\hat{B}}^{+}~\mbox{and}~\frac{\mu(\tau^{-n}(A))}{\mu(A)} >M,~\mbox{for each}~A \in \mathbb{A}_{\hat{C}}^{+}.$$
Let $g \in \hat{S}_{L^{pq}(X)}$ be simple function i.e., $g=\sum g(A_{n})\chi_{A_{n}},$ where $A_{n}$ is atom with measure space $(X,\mathbb{A},\mu)$ is atomic. Write $g=g_{A_{B}^{+}}+g_{A_{C}^{+}} \in \hat{S}_{L^{pq}(X)}$. Then 
$$g_{A_{B}^{+}}=\sum_{A_{n} \in A_{B}^{+} }g(A_{n})\chi_{A_{n}}~\mbox{and}~g_{A_{C}^{+}}=\sum_{A_{n} \in A_{C}^{+} }g(A_{n})\chi_{A_{n}}.$$
Since $||g||_{pq}^{p}=||g_{A_{B}^{+}}||_{pq}^{p}+||g_{A_{C}^{+}}||_{pq}^{p}=1.$ So, either $||g_{A_{B}^{+}}||_{pq}^{p}\geq \frac{1}{2}~\mbox{or}~ ||g_{A_{C}^{+}}||_{pq}^{p} \geq \frac{1}{2}.$\\
In the very first case, for each $n \geq \bar{n}$, $\lambda>0~\mbox{and for} ~A_{n} \in A_{B}^{+} $, we have
\begin{eqnarray*}
\mu_{C_{\tau}^{-n}g}(\lambda) &=& \displaystyle\sum_{n \geq \bar{n},|g(A_{n})|>\lambda}\mu(\tau^{n}(A_{n})) \\
                             &\geq& \displaystyle\sum_{n \geq \bar{n},|g(A_{n})|>\lambda} M \mu(A_{n}) \\
														&=& M \displaystyle\sum_{n \geq \bar{n},|g(A_{n})|>\lambda}  \mu(A_{n}).
\end{eqnarray*}
Therefore, $||C_{\tau}^{-n}g||_{pq}^{p} \geq M ||g||_{pq}^{p} > \frac{M}{2}.$\\
Further, for each $n \geq \bar{n}$, $\lambda>0~\mbox{and for} ~A_{n} \in A_{C}^{+} $, we have
\begin{eqnarray*}
\mu_{C_{\tau}^{n}g}(\lambda) &=& \displaystyle\sum_{n \geq \bar{n},|g(A_{n})|>\lambda}\mu(\tau^{-n}(A_{n})) \\
                             &\geq& \displaystyle\sum_{n \geq \bar{n},|g(A_{n})|>\lambda} M \mu(A_{n}) \\
														&=& M \displaystyle\sum_{n \geq \bar{n},|g(A_{n})|>\lambda}  \mu(A_{n}).
\end{eqnarray*}
Therefore, $||C_{\tau}^{n}g||_{pq}^{p} \geq M ||g||_{pq}^{p} > \frac{M}{2}.$\\
From above, it follows that $ \hat{S}_{L^{pq}(X)}=\hat{B} \cup \hat{C},$ we have 
$$\hat{B}=\{ g \in \hat{S}_{L^{pq}(X)}:||g_{A_{B}^{+}}||_{pq}^{p}\geq \frac{1}{2} \}~\mbox{and}~\hat{C}=\{ g \in \hat{S}_{L^{pq}(X)}:||g_{A_{C}^{+}}||_{pq}^{p}\geq \frac{1}{2} \}.$$
Thus result is proved for simple functions. Since simple functions are dense in $L^{pq}(X)$, it follows that $S_{L^{pq}(X)}=B\cup C$ with $$\displaystyle\lim_{n \to \infty}||C_{\tau}^{n}g ||_{pq}=\infty,~\mbox{uniformly on B}~\mbox{and}~\displaystyle\lim_{n \to \infty}||C_{\tau}^{-n}g ||_{pq}=\infty,~\mbox{uniformly on C}.$$
\end{proof}


\end{document}